\documentclass[11pt]{amsart}
\usepackage{amssymb}
\usepackage{amsfonts}
\usepackage{color}

\usepackage{xy}
\xyoption{all}

\newtheorem{theorem}{Theorem}[section]
\newtheorem{proposition}[theorem]{Proposition}
\newtheorem{lemma}[theorem]{Lemma}
\newtheorem{corollary}[theorem]{Corollary}
\newtheorem{definition}[theorem]{Definition}

\theoremstyle{remark}
\newtheorem{remark}[theorem]{Remark}

\renewenvironment{proof}{{\noindent\bf Proof.}}{\hfill $\Box$\par\vskip3mm}

\newcommand{\im}{{\rm Im}\,}

\newcommand{\Hom}{{\rm Hom}}

\newcommand{\Cc}{\mathcal{C}}

\newcommand{\Mm}{\mathcal{M}}

\newcommand{\Tt}{\mathcal{T}}
\newcommand{\Ss}{\mathcal{S}}
\def\RR{{\mathbb R}}

\def\NN{{\mathbb N}}

\begin{document}
\title[Representation Theory and Hopf Algebras]{Generalized Frobenius Algebras and the Theory of Hopf Algebras}%

\begin{abstract}
"Co-Frobenius" coalgebras were introduced as dualizations of Frobenius algebras. Recently, it was shown in \cite{I} that they admit left-right symmetric characterizations analogue to those of Frobenius algebras: a coalgebra $C$ is co-Frobenius if and only if it is isomorphic to its rational dual. We consider the more general quasi-co-Frobenius (QcF) coalgebras; in the first main result we show that these also admit symmetric characterizations: a coalgebra is QcF if it is weakly isomorphic to its (left, or equivalently right) rational dual $Rat(C^*)$, in the sense that certain coproduct or product powers of these objects are isomorphic. These show that QcF coalgebras can be viewed as generalizations of both co-Frobenius coalgebras and Frobenius algebras. Surprisingly, these turn out to have many applications to fundamental results of Hopf algebras. The equivalent characterizations of Hopf algebras with left (or right) nonzero integrals as left (or right) co-Frobenius, or QcF, or semiperfect or with nonzero rational dual all follow immediately from these results. Also, the celebrated uniqueness of integrals follows at the same time as just another equivalent statement. Moreover, as a by-product of our methods, we observe a short proof for the bijectivity of the antipode of a Hopf algebra with nonzero integral. This gives a purely representation theoretic approach to many of the basic fundamental results in the theory of Hopf algebras.
\end{abstract}

\author{Miodrag Cristian Iovanov\\}%$^*$}
\thanks{$^*$ The author was partially supported by the contract nr. 24/28.09.07 with UEFISCU "Groups, quantum groups, corings and representation theory" of CNCIS, PN II (ID\_1002)\\THIS PAPER WAS COMPILED FROM A TEX-SUBMISSION TO arXiv}

%\thanks{$^*$ This paper was written within the frame of the bilateral Flemish-Romanian project "New Techniques in Hopf Algebra Theory and Graded Rings"}
%\thanks{$^2$}
\date{}
\maketitle

\section*{Introduction}
%\noindent

A $K$ algebra $A$ over a field $K$ is called Frobenius if $A$ is isomorphic to $A^*$ as right $A$-modules. This is equivalent to there being an isomorphism of left $A$-modules between $A$ and $A^*$. This is the modern algebra language formulation for an old question posed by Frobenius. Given a finite dimensional algebra with a basis $x_1,\dots,x_n$, the left multiplication by an element $a$ induces a representation $A\mapsto End_K(A)=M_n(K)$, $a\mapsto (a_{ij})_{i,j}$ ($a_{ij}\in K$), where $a\cdot x_i=\sum\limits_{j=1}^na_{ij}x_j$. Similarly, the right multiplication produces a matrix $a_{ij}^\prime$ by writing $x_i\cdot a=\sum\limits_{j=1}^na_{ji}^\prime x_j$, $a_{ij}^\prime\in K$, and this induces another representation $A\ni a\mapsto (a_{ij}^\prime)_{i,j}$. Frobenius' problem came as the natural question of when the two representations are equivalent. Frobenius algebras occur in many different fields of mathematics, such as topology (the cohomology ring of a compact manifold with coefficients in a field is a Frobenius algebra by Poincar${\rm \acute{e}}$ duality), topological quantum field theory (there is a one-to-one correspondence between 2-dimensional quantum field theories and commutative Frobenius algebras; see \cite{Ab}), Hopf algebras (a finite dimensional Hopf algebra is a Frobenius algebra), and Frobenius algebras have subsequently developed into a research subfield of algebra.

\vspace{.5cm}

Co-Frobenius coalgebras were first introduced by Lin in \cite{L} as a dualization of Frobenius algebras. A coalgebra is left (right) co-Frobenius if there is a monomorphism of left (right) $C^*$-modules $C\subseteq C^*$. However, unlike the algebra case, this concept is not left-right symmetric, as an example in \cite{L} shows. Nevertheless, in the case of Hopf algebras, it was observed that left co-Frobenius implies right co-Frobenius. Also, a left (or right) co-Frobenius coalgebra can be infinite dimensional, while a Frobenius algebra is necessarily finite dimensional. Co-Frobenius coalgebras are coalgebras that are both left and right co-Frobenius. It recently turned out that this notion of co-Frobenius has a nice characterization that is analogue to the characterizations of Frobenius algebras and is also left-right symmetric: a coalgebra $C$ is co-Frobenius if it is isomorphic to its left (or equivalently to its right) rational dual $Rat({}_C^*C^*)$ (equivalently $C\simeq Rat(C^*_{C^*})$; see \cite{I}). This also allowed for a categorical characterization which is again analogue to a characterization of Frobenius algebras: an algebra $A$ is Frobenius iff the functors $\Hom_A(-,A)$ ("the $A$-dual functor") and $\Hom_K(-,K)$ ("the $K$-dual functor") are naturally isomorphic. Similarly, a coalgebra is co-Frobenius if the $C^*$-dual $\Hom_{C^*}(-,C^*)$ and the $K$-dual $\Hom_K(-,K)$ functors are isomorphic on comodules. If a coalgebra $C$ is finite dimensional then it is co-Frobenius if and only if $C$ is Frobenius, showing that the co-Frobenius coalgebras (or rather their dual) can be seen as the infinite dimensional generalization of Frobenius algebras. 

\vspace{.5cm}

Quasi-co-Frobenius (QcF) coalgebras were introduced in \cite{NT1} (further investigated in \cite{NT2}), as a natural dualization of quasi-Frobenius algebras (QF algebras), which are algebras that are self-injective, cogenerators and artinian to the left, equivalently, all these conditions to the right. However, in order to allow for infinite dimensional QcF coalgebras (and thus obtain more a general notion), the definition was weaken to the following: a coalgebra is said to be left (right) QcF if it embeds in $\coprod\limits_IC^*$ (a direct coproduct of copies of $C^*$) as left (right) $C^*$-modules. These coalgebras were shown to bear many properties that were the dual analogue of the properties of QF algebras. Again, this turned out not to be a left-right symmetric concept, and QcF coalgebras were introduced to be the coalgebras which are both left and right QcF. Our first goal is to note that the results and techniques of \cite{I} can be extended and applied to obtain a symmetric characterization of these coalgebras. In the first main result we show that a coalgebra is QcF if and only if $C$ is "weakly" isomorphic to $Rat({}_{C^*}C^*)$ as left $C^*$-modules, in the sense that some (co)product powers of these objects are isomorphic, and this is equivalent to asking that $C^*$ is "weakly" isomorphic to $Rat(C^*_{C^*})$ (its right rational dual) as right $C^*$-modules. In fact, it is enough to have an isomorphism of countable powers of these objects. This also allows for a nice categorical characterization, which states that $C$ is QcF if and only if the above $C^*$-dual and $K$-dual functors are (again) "weakly" isomorphic. Besides realizing QcF coalgebras as a left-right symmetric concept which is a generalization of both Frobenius algebras, co-Frobenius co-algebras and co-Frobenius Hopf algebras, we note that this also provides this characterization of finite dimensional quasi-Frobenius algebras: $A$ is QF iff $A$ and $A^*$ are weakly isomorphic in the above sense, equivalently, $\coprod_\NN A\simeq \coprod_\NN A^*$. 

\vspace{.5cm}

Thus these results give a nontrivial generalization of Frobenius algebras and of quasi-Frobenius algebras, and the algebras arising as dual of QcF coalgebras are entitled to be called Generalized Frobenius Algebras, or rather Generalized QF Algebras.

\vspace{.5cm}

These turn out to have a wide range of applications to Hopf algebras. In the theory of Hopf algebras, some of the first fundamental results were concerned with the characterization of Hopf algebras having a nonzero integral. These are in fact generalizations of well known results from the theory of compact groups. Recall that if $G$ is a (locally) compact group, then there is a unique left invariant (Haar) measure and an associated integral $\int$. Considering the algebra $R_c(G)$ of continuous representative functions on $G$, i.e. functions $f:G\rightarrow \RR$) such that there are $f_i,g_i:G\rightarrow K$ for $i=1,n$ with $f(xy)=\sum\limits_{i=1}^n f_i(x)g_i(y)$, then this becomes a Hopf algebra with multiplication given by the usual multiplication of functions, comultiplication given by $f\mapsto \sum\limits_{i=1}^nf_i\otimes g_i$ and antipode $S$ given by the composition with the taking of inverses $S(f)(x)=f(x^{-1})$. Then, the integral $\int$ of $G$ restricted to $R_c(G)$ becomes an element of $R_c(G)^*$ that has the following property: $\alpha\cdot\int=\alpha({\mathbf 1})\int$, with ${\mathbf 1}$ being the constant 1 function. Such an element in a general Hopf algebra is called a left integral, and Hopf algebras (quantum groups) having a nonzero left integral can be viewed as ("quantum") generalizations of compact groups (the Hopf Algebra can be thought of as the algebra of continuous representative functions on some abstract quantum space). Among the first the fundamental results in Hopf algebras was (were) the fact(s) that the existence of a left integral is equivalent to the existence of a right integral, and these are equivalent to the (co)representation theoretic properties of the underlying coalgebra of $H$ of being left co-Frobenius, right co-Frobenius, left (or right) QcF, or having nonzero rational dual. These were results obtained in several initiating research papers on Hopf algebras \cite{LS, MTW, R, Su, Sw1}. Then the natural question of whether the integral in a Hopf algebra is unique arose (i.e. the space of left integrals $\int_l$ or that of right integrals $\int_r$ is one dimensional), which would generalize the results from compact groups. The answer to this question turned out positive, as it was proved by Sullivan in \cite{Su}; alternate proofs followed afterwards (see \cite{Ra, St}). Another very important result is that of Radford, who showed that the antipode of a Hopf algebra with nonzero integral is always bijective. 

\vspace{.5cm}

We re-obtain all these results as a byproduct of our co-representation theoretic results and generalizations of Frobenius algebras; they will turn out to be an easy application of these general results. We also note a very short proof of the bijectivity of the antipode by constructing a certain derived comodule structure on $H$, obtained by using the antipode and the so called distinguished grouplike element of $H$, and the properties of the comodule $H^H$. The only way we need to use the Hopf algebra structure of $H$ is through the classical Fundamental theorem of Hopf modules which gives an isomorphism of $H$-Hopf modules $\int_l\otimes H\simeq Rat({}_{H^*}H^*)$; however, we will only need to use that this is a isomorphism of comodules. We thus find almost purely representation theoretic proofs of all these classical fundamental results from the theory of Hopf algebras, which become immediate easy applications of the more general results on the "generalized Frobenius algebras". Thus, the methods and results in this paper are also intended to emphasize the potential of these representation theoretic approaches. 

\vspace{.5cm}

\section{Quasi-co-Frobenius Coalgebras}

Let $C$ be a coalgebra over a field $K$. We denote by $\Mm^C$ (respectively ${}^C\Mm$) the category of right (left) $C$-comodules and by ${}_{C^*}\Mm$ (respectively $\Mm_{C^*}$) the category of left (right) $C^*$-modules. We use the simplified Sweedler's$ \sigma$-notation for the comultiplication $\rho:M\rightarrow M\otimes C$ of a $C$-comodule $M$, $\rho(m)=m_0\otimes m_1$. We will always use the inclusion of categories $\Mm^C\hookrightarrow {}_{C^*}\Mm$, where the left $C^*$-module structure on $M$ is given by $c^*\cdot m=c^*(m_1)m_0$. 

Let ${\mathcal S}$ be a set of representatives for the types of isomorphism of simple left $C$-comodules and ${\mathcal T}$ be a set of representatives for the simple right comodules. It is well known that we have an isomorphism of left $C$-comodules (equivalently right $C^*$-modules) $C\simeq \bigoplus\limits_{S\in \Ss}E(S)^{n(S)}$, where $E(S)$ is an injective envelope of the left $C$-comodule $S$ and $n(S)$ are positive integers. Similarly, $C\simeq \bigoplus\limits_{T\in\Tt}E(T)^{p(T)}$ in $\Mm^C$, with $p(T)\in\NN$ (we use the same notation for envelopes of left modules and for those of right modules, as it will always be understood from the context what type of modules we refer to). Also $C^*\simeq \prod\limits_{S\in\Ss}E(S)^*$ in ${}_{C^*}\Mm$ and $C^*\simeq \prod\limits_{T\in\Tt}E(T)^*$ in $\Mm_{C^*}$. We refer the reader to \cite{A}, \cite{DNR} or \cite{Sw} for these results and other basic facts of coalgebras. We will use the finite topology on duals of vector spaces: given a vector space $V$, this is the linear topology on $V^*$ that has a basis of neighbourhoods of $0$ formed by the sets $F^\perp=\{f\in V^*\mid\,f\vert_F=0\}$ for finite dimensional subspaces $F$ of $V$. We also write $W^\perp=\{x\in V\mid f(x)=0,\,\forall f\in W\}$ for subsets $W$ of $V^*$. Any topological reference will be with respect to this topology. \\
For a module $M$, we convey to write $M^{(I)}$ for the coproduct (direct sum) of $I$ copies of $M$ and $M^I$ for the product of $I$ copies of $M$. 
We recall the following definition from \cite{NT1}

\begin{definition}
A coalgebra $C$ is called right (left) quasi-co-Frobenius, or shortly right QcF coalgebra, if there is a monomorphism $C\hookrightarrow (C^*)^{(I)}$ of right (left) $C^*$-modules. $C$ is called QcF coalgebra if it is both a left and right QcF coalgebra.
\end{definition}

Recall that a coalgebra $C$ is called right semiperfect if the category $\Mm^C$ of right $C$-comodules is semiperfect, that is, every right $C$-comodule has a projective cover. This is equivalent to the fact that $E(S)$ is finite dimensional for all $S\in\Ss$ (see \cite{L}). In fact, this is the definition we will need to use. For convenience, we also recall the following very useful results on injective (projective) comodules, the first one originally given in \cite{D} Proposition 4, p.34 and the second one being Lemma 15 from \cite{L}:

\vspace*{.3cm}
%{\cite[\bf Corollary 2.1.19]{DNR}} 
{\cite[\bf Proposition 4]{D}} {\emph Let $Q$ be a finite dimensional right $C$-comodule. Then $Q$ is injective (projective) as a left $C^*$-module if and only if it is injective (projective) as right $C$-comodule.}

\vspace{.3cm}
%{\cite[\bf Corollary 2.1.20]{DNR}} 
{\cite[\bf Lemma 15]{L}} {\emph Let $M$ be a finite-dimensional right $C$-comodule. Then $M$ is an injective right $C$-comodule if and only if $M^*$ is a projective left $C$-comodule.}

\vspace{.3cm}
We note the following proposition that will be useful in what follows; (i)$\Leftrightarrow$(ii) was given in \cite{NT1} and our approach also gives here a different proof, along with the new characterizations.

\begin{proposition}\label{p1}
Let $C$ be a coalgebra. Then the following assertions are equivalent:\\
(i) $C$ is a right QcF coalgebra.\\
(ii) $C$ is a right torsionless module, i.e. there is a monomorphism $C\hookrightarrow (C^*)^I$.\\
(iii) There exist a dense morphism $\psi:C^{(I)}\rightarrow C^*$, that is, the image of $\psi$ is dense in $C^*$.\\
(iv) $\forall S\in\Ss$, $\exists T\in \Tt$ such that $E(S)\simeq E(T)^*$.
\end{proposition}
\begin{proof}
(i)$\Rightarrow$(ii) obvious.\\
(ii)$\Leftrightarrow$(iii) Let $\varphi:C\rightarrow (C^*)^I\simeq (C^{(I)})^*$ be a monomorphism of right $C^*$-modules. Let $\psi:C^{(I)}\rightarrow C^*$ be defined by $\psi(x)(c)=\varphi(c)(x)$. It is straightforward to see that the fact that $\varphi$ is a morphism of left $C^*$-modules is equivalent to $\psi$ being a morphism in $\Mm_{C^*}$, and that $\varphi$ injective is equivalent to $(\im \psi)^\perp=0$, which is further equivalent to $\im \psi$ is dense in $C^*$ (for example, by \cite{DNR} Corollary 1.2.9).\\
(ii),(iii)$\Rightarrow$(iv) As $\im\psi\subseteq Rat({}_{C^*}C^*)$, $Rat({}_{C^*}C^*)$ is dense in $C^*$, so $C$ is right semiperfect by Proposition 3.2.1 \cite{DNR}. Thus $E(S)$ is finite dimensional $\forall S\in \Ss$. Also by (ii) there is a monomorphism $\iota:E(S)\hookrightarrow \prod\limits_{j\in J}E(T_j)^*$ with $T_j\in \Tt$, and as ${\rm dim}\,E(S)<\infty$ there is a monomorphism to a finite direct sum $E(S)\hookrightarrow \prod\limits_{j\in F}E(T_j)^*$ ($F$ finite, $F\subseteq J$). Indeed, if $p_j$ are the projections of $\prod\limits_{j\in J}E(T_j)^*$, then note that $\bigcap\limits_{j\in J} \ker p_j\circ \iota=0$, so there must be $\bigcap\limits_{j\in F}\ker p_j\circ\iota=0$ for a finite $F\subseteq J$. Then $E(S)$ is injective also as right $C^*$-modules (see for example  \cite{DNR}, Corollary 2.4.19), and so $E(S)\oplus X=\bigoplus\limits_{j\in F}E(T_j)^*$ for some $X$. By \cite[Lemma 1.4]{I}, the $E(T_j)^*$'s are local indecomposable, and as they are also cyclic projective we eventually get $E(S)\simeq E(T_j)^*$ for some $j\in F$. This can be easily seen by noting that there is at least one nonzero morphism $E(S)\hookrightarrow E(S)\oplus X=\bigoplus\limits_{j\in F}E(T_j)^*\rightarrow\bigoplus\limits_{j\in F}T_j^*\rightarrow S_k$ (one looks at Jacobson radicals) and this projection then lifts to a morphism $f:E(S)\rightarrow E(T_k)^*$ as $E(S)$ is obviously projective; this has to be surjective since $E(T_k)^*$ is cyclic local, and then $f$ splits; hence $E(S)\simeq E(T_k)^*\oplus Y$ with $Y=0$ as $E(S)$ is indecomposable.\\
(iv)$\Rightarrow$(i) Any isomorphism $E(S)\simeq E(T)^*$ implies $E(S)$ finite dimensional because then $E(T)^*$ is cyclic rational; therefore $E(T)\simeq E(S)^*$. Thus for each $S\in \Ss$ there is exactly one $T\in \Tt$ such that $E(S)\simeq E(T)^*$. If $\Tt'$ is the set of these $T$'s, then:
\begin{eqnarray*}
C & \simeq & \bigoplus\limits_{S\in\Ss}E(S)^{n(S)} \hookrightarrow \bigoplus\limits_{S\in\Ss}E(S)^{(\NN)} \simeq \bigoplus\limits_{T\in\Tt'\subseteq \Tt}(E(T)^*)^{(\NN)}\\
& \hookrightarrow & (\bigoplus\limits_{T\in \Tt}(E(T)^*)^{p(T)})^{(\NN)} \subseteq (\prod\limits_{T\in \Tt}(E(T)^*)^{p(T)})^{(\NN)}=C^*{}^{(\NN)}
\end{eqnarray*}
\end{proof}

From the above proof, we see that when $C$ is right QcF, the $E(S)$'s are finite dimensional projective for $S\in\Ss$, and we also conclude the following result already known from \cite{NT1} (in fact these conditions are even equivalent); see also \cite[Theorem 3.3.4]{DNR}.

\begin{corollary}\label{c1}
If $C$ is right QcF, then $C$ is also right semiperfect and projective as right $C^*$-module.
\end{corollary}

We also immediately conclude the following

\begin{corollary}\label{c2}
A coalgebra $C$ is QcF if and only if the application $$\{E(S)\mid S\in \Ss\}\ni Q\mapsto Q^*\in\{E(T)\mid T\in\Tt\}$$ is well defined and bijective.
\end{corollary}

\begin{definition}\label{def}
(i) Let $\Cc$ be a category having products. We say that $M,N\in \Cc$ are weakly $\pi$-isomorphic if and only if there are some sets $I,J$ such that $M^I\simeq N^J$; we write $M\stackrel{\pi}{\sim}N$.\\
(ii) Let $\Cc$ be a category having coproducts. We say that $M,N\in \Cc$ are weakly $\sigma$-isomorphic if and only if there are some sets $I,J$ such that $M^{(I)}\simeq N^{(J)}$; we write $M\stackrel{\sigma}{\sim}N$.
\end{definition}

The study of objects of a (suitable) category $\Cc$ up to $\pi$(respectively $\sigma$)-isomorphism is the study of the localization of $\Cc$ with respect to the class of all $\pi$(or $\sigma$)-isomorphisms.

Recall that in the category ${}^C\Mm$ of left comodules, coproducts are the usual direct sums of (right) $C^*$-modules and the product $\prod\limits^C$ is given, for a family of comodules $(M_i)_{i\in I}$, by $\prod\limits_{i\in I}^CM_i=Rat(\prod\limits_{i\in I}M_i)$. 

%For convenience, we also recall the following very usefull results on injective (projective) comodules, the first one originally given in \cite{D} Proposition 4, p.34 and the second one being Lemma 15 from \cite{L}:

%\vspace*{.3cm}
%{\cite[\bf Corollary 2.1.19]{DNR}} {\emph Let $Q$ be a finite dimensional right $C$-comodule. Then $Q$ is injective (projective) as a left $C^*$-module if and only if it is injective (projective) as right $C$-comodule.}

%\vspace{.3cm}
%{\cite[\bf Corollary 2.1.20]{DNR}} {\emph Let $M$ be a finite-dimensional right $C$-comodule. Then $M$ is an injective right $C$-comodule if and only if $M^*$ is a projective left $C$-comodule.}

\vspace{.3cm}
For easy future reference, we introduce the following conditions:

(C1) $C\stackrel{\sigma}{\sim}Rat(C^*_{C^*})$ in ${}^C\Mm$ (or equivalently, in $\Mm_{C^*}$).\\
(C2) $C\stackrel{\pi}{\sim}Rat(C^*_{C^*})$ in ${}^C\Mm$.\\
(C3) $Rat(C^I)\simeq Rat(C^*{}^J)$ for some sets $I,J$.\\
(C2') $C\stackrel{\pi}{\sim}Rat(C^*_{C^*})$ in $\Mm_{C^*}$.

\begin{lemma}\label{l1}
Either one of the conditions (C1), (C2), (C3), (C2') implies that $C$ is QcF (both left and right).
\end{lemma}
\begin{proof}
Obviously (C2')$\Rightarrow$(C2). In all of the above conditions one can find a monomorphism of right $C^*$-modules $C\hookrightarrow (C^*)^J$, and thus $C$ is right QcF. Then each $E(S)$ for $S\in\Ss$ is finite dimensional and projective  by Corollary \ref{c1}. We first show that $C$ is also left semiperfect, along the same lines as the proofs of \cite{I}, Proposition 2.1 and \cite{I} Proposition 2.6. For sake of completeness, we include a short version of these arguments here. Let $T_0\in\Tt$ and assume, by contradiction, that $E(T_0)$ is infinite dimensional. We first show that $Rat(E(T_0)^*)=0$. Indeed, assume otherwise. Then, since $C^*=\prod\limits_{T\in\Tt}E(T)^*{}^{p(T)}$ and $C=\bigoplus\limits_{S\in \Ss}E(S)^{n(S)}$ as right $C^*$-modules, it is straightforward to see that either one of conditions (C1-C3) implies that $Rat(E(T_0)^*)$ is injective as left comodule, as a direct summand in an injective comodule. Thus, as $Rat(E(T_0)^*)\neq 0$, there is a monomorphism $E(S)\hookrightarrow Rat(E(T_0)^*)\subseteq E(T_0)^*$ for some indecomposable injective $E(S)$ ($S\in \Ss$). This shows that $E(S)$ is a direct summand in $E(T_0)^*$, since $E(S)$ is injective also as right $C^*$-module (by the above cited \cite[Proposition 4]{D}). But this is a contradiction since $E(S)$ is finite dimensional and $E(T_0)^*$ is indecomposable by \cite[Lemma 1.4]{I} and ${\rm dim}E(T_0)^*=\infty$.\\
Next, use \cite[Proposition 2.3]{I}  to get an exact sequence
$$0\rightarrow H \rightarrow E=\bigoplus\limits_{\alpha\in A}E(S_\alpha)^*\rightarrow E(T)\rightarrow 0$$
with $S_\alpha\in\Ss$. Since the $E(S_\alpha)^*$'s are injective in ${}_{C^*}\Mm$ by \cite[Lemma 15]{L}, we may assume, by \cite[Proposition 2.4]{I} that $H$ contains no nonzero injective right comodules. For some $\beta\in A\neq \emptyset$, put $E'=\bigoplus\limits_{\alpha\in A\setminus\{\beta\}}E(S_\alpha)^*$. Then one sees that $H+E'=E$ (otherwise, since there is an epimorphism $E(T)=\frac{E}{H}\rightarrow \frac{E}{H+E'}$, the finite dimensional rational right $C^*$-module $\left(\frac{E}{H+E'}\right)^*$ would be a nonzero rational submodule of $E(T)^*$), and this provides an epimorphism $H\rightarrow \frac{H}{H\cap E'}\simeq \frac{H+E'}{E'}\simeq E(S_\beta)^*$. But $E(S_\beta)^*$ is projective, so this epimorphism splits, and this comes in contradiction with the assumption on $H$ (the $E(S_\beta)^*$'s are injective in ${}_{C^*}\Mm$).
\end{proof}

Now, we note that if a coalgebra $C$ is QcF, then all the conditions (C1)-(C3) are fulfilled. Indeed, we have that each $E(S)$ ($S\in\Ss$) is isomorphic to exactly one $E(T)^*$ ($T\in\Tt$) and actually all $E(T)^*$'s are isomorphic to some $E(S)$. Then:\\
%(C1) 
(C1)
\begin{eqnarray*}
\,\,\,\,\,C^{(\NN)} & = & (\bigoplus\limits_{S\in\Ss}E(S)^{n(S)})^{(\NN)}=\bigoplus\limits_{S\in\Ss}E(S)^{(\NN)}=\bigoplus\limits_{T\in\Tt}E(T)^*{}^{(\NN)} \\
 & = & \bigoplus\limits_{T\in\Tt}E(T)^*{}^{(p(T)\times\NN)}=\bigoplus\limits_{T\in\Tt}(E(T)^*{}^{p(T)})^{(\NN)}=(Rat C^*)^{(\NN)}
\end{eqnarray*}
where we use that $Rat(C^*)=\bigoplus\limits_{T\in\Tt}E(T)^*{}^{p(T)}$ as right $C^*$-modules for left and right semiperfect coalgebras (see \cite[Corollary 3.2.17]{DNR})\\
%(C2)
(C2)
\begin{eqnarray*}
\,\,\,\,\,\prod\limits_{\NN}^CC & = & Rat(C^\NN)=\prod\limits_\NN^C\bigoplus\limits_{S\in\Ss}E(S)^{(n(S))}\\
& = & \prod\limits_\NN^C\prod\limits_{S\in\Ss}^CE(S)^{n(S)}\,\,\,(*)\\%\,\,\,{\rm by\,}(\cite{I})\,{\rm as\,}C\,{\rm is\,quasifinite}
& = & \prod\limits_{S\in\Ss}^CE(S)^{n(S)\times\NN}=\prod\limits_{S\in\Ss}^CE(S)^\NN=\prod\limits_{T\in\Tt}^CE(T)^*{}^\NN\\
& = & \prod\limits_{T\in \Tt}^CE(T)^*{}^{\NN\times p(T)}=\prod\limits_\NN^C\prod\limits_{T\in\Tt}^CE(T)^*{}^{p(T)}\\
& = & \prod\limits_\NN^CRat(\prod\limits_{T\in\Tt}(E(T)^{p(T)})^*)=\prod\limits_\NN^CRat((\bigoplus\limits_{T\in T}E(T)^{p(T)})^*)\\
& = &\prod\limits_\NN^CRat(C^*)\,\,\,
\end{eqnarray*}
where for (*) we have used \cite[Lemma 2.5]{I} and the fact that $E(T)^*$ are all rational since $E(T)$ are finite dimensional in this case (the product in the category of left comodules is understood whenever $\prod\limits^C$ is written); also\\
(C3) holds because $Rat(C^\NN)=\prod\limits_{T\in\Tt}^CE(T)^*{}^\NN$ by the computations in lines 1 and 3 in the above equation and because
\begin{eqnarray*}
\,\,\,\,\,Rat(C^*{}^\NN)=Rat(\prod\limits_\NN\prod\limits_{T\in\Tt}E(T)^*{}^{p(T)})=\prod\limits_{T\in\Tt}^CE(T)^*{}^{p(T)\times\NN}=\prod\limits_{T\in\Tt}^CE(T)^*{}^\NN
\end{eqnarray*}

Combining all of the above we obtain the following nice symmetric characterization which extends the one of co-Frobenius coalgebras from \cite{I} and those of co-Frobenius Hopf algebras and Frobenius Algebras.%the above with Proposition \ref{p1} and Lemma {}

\begin{theorem}\label{TQ}
Let $C$ be a coalgebra. Then the following assertions are equivalent.\\
(i) $C$ is a QcF coalgebra.\\
(ii) $C\stackrel{\sigma}{\sim}Rat(C^*_{C^*})$ or $C\stackrel{\pi}{\sim}Rat(C^*_{C^*})$ in ${}^C\Mm$ or $Rat(C^I)\simeq Rat(C^*{}^J)$ in ${}^C\Mm$ (or $\Mm_{C^*}$) for some sets $I,J$.\\
(iii) $C^{(\NN)}\simeq (Rat(C^*))^{(\NN)}$ or $\prod\limits_\NN^C C\simeq \prod\limits_\NN^C Rat(C^*)$ or $Rat(C^\NN)\simeq Rat(C^*{}^\NN)$ as left $C$-comodules (right $C^*$-modules)\\
(iv) The left hand side version of (i)-(iii).\\
(v) The association $Q\mapsto Q^*$ determines a duality between the finite dimensional injective left comodules and finite dimensional injective right comodules.
\end{theorem}

\subsection{Categorical characterization of QcF coalgebras}

We give now a characterization similar to the functorial characterizations of co-Frobenius coalgebras and of Frobenius algebras.
For a set $I$ let $\Delta_I:{}^C\Mm\longrightarrow ({}^C\Mm)^I$ be the diagonal functor and let $F_I$ be the composition functor
$$F_I:\,{}^C\Mm\stackrel{\Delta_I}{\longrightarrow}({}^C\Mm)^I\stackrel{\bigoplus\limits_I}{\longrightarrow}{}^C\Mm$$ 
that is $F_I(M)=M^{(I)}$ for any left $C$-comodule $M$.

\begin{theorem}
Let $C$ be a coalgebra. Then the following assertions are equivalent:\\
(i) $C$ is QcF.\\
(ii) The functors $\Hom_{C^*}(-,C^*)\circ F_I$ and $\Hom(-,K)\circ F_J$ from ${}^C\Mm=Rat({}_{C^*}\Mm)$ to ${}_{C^*}\Mm$ are naturally isomorphic for some sets $I,J$.\\
(iii) The functors $\Hom_{C^*}(-,C^*)\circ F_\NN$ and $\Hom(-,K)\circ F_\NN$ are naturally isomorphic.
\end{theorem}
\begin{proof}
Since for any left comodule $M$, there is a natural isomorphism of left $C^*$-modules $\Hom_{C^*}(M,C)\simeq \Hom(M,K)$, then for any sets $I,J$ and any left $C$-comodule $M$ we have the following natural isomorphisms:
$$\Hom(M^{(I)},K)\simeq \Hom_{C^*}(M^{(I)},C)\simeq \Hom_{C^*}(M,C^I)\simeq \Hom_{C^*}(M,Rat(C^I))$$
$$\Hom_{C^*}(M^{(J)},C^*)\simeq \Hom_{C^*}(M,(C^*)^J)\simeq \Hom(M,Rat(C^*)^J)$$
Therefore, by the Yoneda Lemma, the functors of (ii) are naturally isomorphic if and only if $Rat(C^I)\simeq Rat(C^*{}^J)$. Thus, by Theorem \ref{TQ} (ii), these functors are isomorphic if and only if $C$ is QcF. Moreover, in this case, by the same theorem the sets $I,J$ can be chosen countable.
\end{proof}

\begin{remark}
The above theorem states that $C$ is QcF if and only if the functors $C^*$-dual $\Hom(-,C^*)$ and $K$-dual $\Hom(-,K)$ from ${}^C\Mm$ to ${}_{}C^*\Mm$ are isomorphic in a "weak" meaning, in the sense that they are isomorphic only on the objects of the form $M^{(\NN)}$ in a way that is natural in $M$, i.e. they are isomorphic on the subcategory of ${}^CM$ consisting of objects $M^{(\NN)}$ with morphisms $f^{(\NN)}$ induced by any $f:M\rightarrow N$. If we consider the category $\Cc$ of functors from ${}^C\Mm$ to ${}_{C^*}\Mm$ with morphisms the \underline{classes} (which are not necessarily sets) of natural transformations between functors, then the isomorphism in (ii) can be restated as $(\Hom_{C^*}(-,C^*))^I\simeq (\Hom(-,K))^J$ in $\Cc$, i.e. the $C^*$-dual and the $K$-dual functors are weakly $\pi$-isomorphic objects of $\Cc$.
\end{remark}
%(ii) The functors $C^*$-dual $\Hom(-,C^*)$ and $K$-dual $\Hom(-,K)$ from ${}^C\Mm$ to ${}_{}C^*\Mm$

\section{Applications to Hopf Algebras}

Before giving the main applications to Hopf algebras, we start with two easy propositions that will contain the main ideas of the applications. First, for a QcF coalgebra $C$, let $\varphi:\Ss\rightarrow\Tt$ be the function defined by $\varphi(S)=T$ if and only if $E(T)\simeq E(S)^*$ as left $C^*$-modules; by the above Corollary \ref{c2}, $\varphi$ is a bijection.

\begin{proposition}\label{p2}
(i) Let $C$ be a QcF coalgebra and $I,J$ sets such that $C^{(I)}\simeq (Rat(C^*))^{(J)}$ as right $C^*$-modules. If one of $I,J$ is finite then so is the other.\\
(ii) Let $C$ be a coalgebra. Then $C$ is co-Frobenius if and only if $C\simeq Rat({}_{C^*}C^*)$ as left $C^*$-modules and if and only if $C\simeq Rat(C^*_{C^*})$ as right $C^*$-modules.
\end{proposition}
\begin{proof}
(i) $C$ is left and right semiperfect (Corollary \ref{c1}), so using again \cite[Corollary 3.2.17]{DNR} we have 
$Rat(C^*_{C^*})=\bigoplus\limits_{T\in\Tt}E(T)^*{}^{p(T)}=\bigoplus\limits_{S\in\Ss}E(S)^{p(\varphi(S))}$ and we get $\bigoplus\limits_{S\in\Ss}E(S)^{n(S)\times I}\simeq \bigoplus\limits_{S\in\Ss}E(S)^{p(\varphi(S))\times J}$. From here, since the $E(S)$'s are indecomposable injective comodules we get an equivalence of sets $n(S)\times I\sim p(\varphi(S))\times J$ (or directly, by evaluating the socle of these comodules). This finishes the proof, as $n(S),p(\varphi(S))$ are finite.\\
(ii) If $C$ is co-Frobenius, $C$ is also QcF and a monomorphism $C\hookrightarrow Rat(C^*_{C^*})$ of right $C^*$-modules implies $\bigoplus\limits_{S\in\Ss}E(S)^{n(S)}\hookrightarrow \bigoplus\limits_{T\in\Tt}E(T)^*{}^{p(T)}\simeq \bigoplus\limits_{S\in\Ss}E(S)^{p(\varphi(S))}$ and we get $n(S)\leq p(\varphi(S))$ for all $S\in\Ss$; similarly, as $C$ is also left co-Frobenius we get $n(S)\geq p(\varphi(S))$ for all $S\in\Ss$. Hence $n(S)=p(\varphi(S))$ for all $S\in\Ss$ and this implies $C=\bigoplus\limits_{S\in\Ss}E(S)^{n(S)}\simeq \bigoplus\limits_{T\in\Tt}E(T)^{p(T)}=Rat(C^*_{C^*})$. Conversely, if $C\simeq Rat(C^*_{C^*})$ by the proof of (i), when $I$ and $J$ have one element we get that $n(S)=p(\varphi(S))$ for all $S\in\Ss$ which implies that we also have $C=\bigoplus\limits_{T\in\Tt}E(T)^{p(T)}\simeq \bigoplus\limits_{S\in\Ss}E(S)^*{}^{n(S)}=Rat({}_{C^*}C^*)$ so $C$ is co-Frobenius.
\end{proof}

The above Proposition \ref{p2} (ii) shows that the results of this paper are a generalization of the results in \cite{I}.

\begin{proposition}\label{p3}
Let $C$ be a (QcF) coalgebra such that $C^k\simeq Rat(C^*_{C^*})$ in $\Mm_{C^*}$ and $C^l\simeq Rat({}_{C^*}C^*)$ in ${}_{C^*}\Mm$, $k,l\in\NN$. Then $C$ is co-Frobenius and $k=l=1$.
\end{proposition}
\begin{proof}
As in the proof of Proposition \ref{p2} we get $k\cdot n(S)=p(\varphi(S))$ for all $S\in\Ss$. Similarly, using $C^l\simeq (Rat({}_{C^*}C^*))$ in ${}_{C^*}\Mm$ we get $l\cdot p(T)=n(\varphi^{-1}(T))$ for $T\in\Tt$ i.e. $n(S)=l\cdot p(\varphi(S))$. These two equations give $k=l=1$ and the conclusion follows as in Proposition \ref{p2} (ii).
\end{proof}

Let $H$ be a Hopf algebra over a basefield $k$. Recall that a left integral for $H$ is an element $\lambda\in H^*$ such that $\alpha\cdot\lambda=\alpha(1)\lambda$, for all $\alpha\in H^*$. The space of left integrals for $H$ is denoted by $\int_l$. The right integrals and the space of right integrals $\int_r$ are defined by analogy. For basic facts on Hopf algebras we refer to \cite{A}, \cite{DNR}, \cite{M} and \cite{Sw}. The Hopf algebra structure will come into play only through a basic Theorem of Hopf algebras, the fundamental theorem of Hopf modules which yields the isomorphism of right $H$-Hopf modules $\int_l\otimes H\simeq Rat({}_{H^*}H^*)$. This isomorphism is given by $t\otimes h\mapsto t\leftharpoondown h=S(h)\rightharpoonup t$, where for $x\in H$, $\alpha\in H^*$, $x\rightharpoonup\alpha$ is defined by $(x\rightharpoonup\alpha)(y)=\alpha(yx)$ and $\alpha\leftharpoondown x=S(x)\rightharpoonup\alpha$. Yet, we will only need that this is an isomorphism of right $H$-comodules (left $H^*$-modules). Similarly, $H\otimes \int_r\simeq Rat(H^*_{H^*})$. %(see \cite[Chapter 5]{DNR})

\begin{theorem}\label{HC}(Lin, Larson, Sweedler, Sullivan)\\
If $H$ is a Hopf algebra, then the following assertions are equivalent.\\
(i) $H$ is a right co-Frobenius coalgebra.\\
(ii) $H$ is a right QcF coalgebra. \\
(iii) $H$ is a right semiperfect coalgebra.\\
(iv) $Rat({}_{H^*}H^*)\neq 0$.\\
(v) $\int_l\neq 0$.\\
(vi) $\dim \int_l=1$.\\
(vii) The left hand side version of the above.
\end{theorem}
\begin{proof}
(i)$\Rightarrow$(ii)$\Rightarrow$(iii) is clear and (iii)$\Rightarrow$(iv) is a property of semiperfect coalgebras (see \cite[Section 3.2]{DNR}). \\
(iv)$\Rightarrow$(v) follows by the isomorphism $\int_l\otimes H\simeq Rat({}_{H^*}H^*)$ and (vi)$\Rightarrow$(v) is trivial.\\
%(v)$\Leftrightarrow$(v')$\Leftrightarrow$(v'') If $I$ is a finite dimensional ideal of $H^*$ it is closed in the finite topology of $H^*$, and $I^\perp$ is a coideal of $H$ of finite codimension. Also, if $K$ is a left coideal of finite codimension in $H$ then $I=K^\perp$ is a finite dimensional left ideal of $H^*$. The equivalence to $$
(v)$\Rightarrow$(i), (vi) and (vii). Since $\int_l\otimes H\simeq Rat({}_{H^*}H^*)$ in $\Mm^H$, we have $H^{(\int_l)}\simeq Rat({}_{H^*}H^*)$ so by Theorem \ref{TQ} $H$ is QcF (both left and right); it then follows that $\int_r\neq 0$ (by the left hand version of (ii)$\Rightarrow$(v)) and $H^{(\int_r)}\simeq Rat(H^*_{H^*})$. We can now apply Propositions \ref{p2} and \ref{p3} to first get that $\dim\int_l<\infty$, $\dim\int_r<\infty$ and then that $H$ is co-Frobenius (both left and right) so (i) and (vii) hold. Again by Proposition \ref{p3} we get that, more precisely, $\dim\int_l=\dim\int_r=1$.
\end{proof}

The following corollary was the initial form of the result proved by Sweedler \cite{Sw1}

\begin{corollary}
The following are equivalent for a Hopf algebra $H$:\\
(i) $H^*$ contains a finite dimensional left ideal.\\
(ii) $H$ contains a left coideal of finite codimension.\\
(iii) $\int_l\neq 0$.\\
(iv) $Rat(H^*)\neq 0$.
\end{corollary}
\begin{proof}
(i)$\Leftrightarrow$(ii) It can be seen by a straightforward computation that there is a bijective correspondence between finite dimensional left ideals $I$ of $H^*$ and coideals $K$ of finite codimension in $H$, given by $I\mapsto K=I^\perp$. Moreover, it follows that any such finite dimensional ideal $I$ of $H^*$ is of the form $I=K^\perp$ with $\dim(H/K)<\infty$, so $I=K^\perp\simeq (H/K)^*$ is then a rational left $H^*$-module, thus $I\subseteq Rat(H^*)$. This shows that (ii)$\Rightarrow$(iv) also holds, while (iii)$\Rightarrow$(ii) is trivial.
\end{proof}

\subsection*{The bijectivity of the antipode}

Let $t$ be a nonzero left integral for $H$. Then it is easy to see that the one dimensional vector space $kt$ is a two sided ideal of $H^*$. Also, by the definition of integrals, $kt\subseteq Rat({}_{H^*}H^*)=Rat(H^*_{H^*})$ (since $H$ is semiperfect as a coalgebra). Thus $kt$ also has a left comultiplication $t\mapsto a\otimes t$, $a\in H$ and then by the coassociativity and counit property for ${}^Hkt$, $a$ has to be a grouplike element. This element is called the \emph{distinguished grouplike} element of $H$. In particular $t\cdot\alpha=\alpha(a)t,\,\forall\alpha\in H^*$. See \cite[Chapter 5]{DNR} for some more details.

\vspace{.5cm}

For any right $H$-comodule $M$ denote ${}^aM$ the left $H$-comodule structure on $M$ with comultiplication
$$M\ni m\mapsto m^a_{-1}\otimes m^a_{0}=aS(m_1)\otimes m_0$$
($S$ denotes the antipode). It is straightforward to see that this defines an $H$-comodule structure. 

\begin{proposition}
The map $p:{}^aH\rightarrow Rat(H^*)$, $p(x)=x\rightharpoonup t$ is a surjective morphism of left $H$-comodules (right $H^*$-modules).
\end{proposition}
\begin{proof}
Since the above isomorphism $H\simeq\int_l\otimes H\simeq Rat(H^*)$ is given by $h\mapsto t\leftharpoondown h=S(h)\rightharpoonup t$, we get the surjectivity of $p$. We need to show that $p(x)_{-1}\otimes p(x)_0=x^a_{-1}\otimes p(x^a_0)$ and for this, having the left $H$-comodule structure of $Rat(H^*)$ in mind, it is enough to show that for all $\alpha\in H^*$, $p(x)_0\alpha(p(x)_{-1})=p(x)\cdot\alpha=\alpha(x^a_{-1})p(x^a_0)$. Indeed, for $g\in H$ we have:
\begin{eqnarray*}
((x\rightharpoonup t)\cdot\alpha)(g) & = & t(g_1x)\alpha(g_2) = t(g_1x_1\varepsilon(x_2))\alpha(g_2) \\
 & = & t(g_1x_1)\alpha(g_2x_2S(x_3)) = t(g_1x_1)(\alpha\leftharpoondown x_3)(g_2x_2) \\
 & = & t((gx_1)_1)(\alpha\leftharpoondown x_2)((gx_1)_2) = (t\cdot (\alpha\leftharpoondown x_2))(gx_1) \\
 & = & (\alpha\leftharpoondown x_2)(a)t(gx_1) \;\;\;\;\; (a{\rm\;is\;the\;distinguished\;grouplike\;of\; }H)\\
 & = & \alpha(aS(x_2))(x_1\rightharpoonup t)(g)
\end{eqnarray*}
and this ends the proof.
\end{proof}

Let $\pi$ be the composition map ${}^aH\stackrel{p}{\longrightarrow}Rat(H^*_{H^*})\stackrel{\sim}{\longrightarrow}H\otimes\int_r\simeq H$, where the isomorphism $H\otimes \int_r\simeq H^*_{H^*}$ is the left analogue of $\int_l\otimes H\simeq Rat({}_{H^*}H^*)$. Since ${}^HH$ is projective in ${}^H\Mm$, this surjective map splits by a morphism of left $H$-comodules $\varphi:H\hookrightarrow {}^aH$, so $\pi\varphi={\rm Id}_H$. Then we can find another proof of:

\begin{theorem}
The antipode of a co-Frobenius Hopf algebra is bijective.
\end{theorem}
\begin{proof}
Since the injectivity of $S$ is immediate from the injectivity of the map $H\ni x\mapsto t\leftharpoondown x\in H^*$, as noticed by Sweedler \cite{Sw1}, we only observe the surjectivity. The fact that $\varphi$ is a morphism of comodules reads $\varphi(x)_{-1}^a\otimes \varphi(x)_0^a=x_1\otimes \varphi(x_2)$, i.e. $aS(\varphi(x)_2)\otimes \varphi(x)_1=x_1\otimes \varphi(x_2)$, and since $a=S(a^{-1})=S^2(a)$, by applying ${\rm Id}\otimes\varepsilon\pi$ we get $S(a^{-1})S(\varphi(x)_2)\varepsilon\pi(\varphi(x)_1)=x_1\varepsilon\pi\varphi(x_2)=x_1\varepsilon(x_2)=x$, so $x=S(\varepsilon\pi(\varphi(x)_1)\varphi(x)_2a^{-1})$.
\end{proof}

%\bigskip\bigskip\bigskip

%\begin{center}
%\sc Acknowledgment
%\end{center}

%\newpage

\vspace*{3mm} 
%\begin{flushright}
%\begin{minipage}{148mm}
\sc\footnotesize
Miodrag Cristian Iovanov\\
University of Bucharest, Faculty of Mathematics, Str. Academiei 14\\ 
RO-010014, Bucharest, Romania\\
and\\
State University of New York (Buffalo)\\
Department of Mathematics, 244 Mathematics Building\\
Buffalo, NY 14260-2900, USA\\
{\it E--mail address}: {\tt
yovanov@gmail.com}\vspace*{3mm}
%\end{minipage}
%\end{flushright}
\end{document}